\documentclass[a4paper,12pt,twoside]{amsart}

\usepackage{amsfonts}
\usepackage{amsfonts,amssymb,amscd}
\usepackage{graphicx}
 
\usepackage[left=2.3cm,top=3.3cm,right=2.3cm,bottom=2.7cm]{geometry}
\usepackage{mathrsfs}
\usepackage{xfrac}
\let\mathcal\mathscr
\usepackage[colorlinks=true,citecolor=blue, urlcolor=blue, breaklinks, linkcolor=black]{hyperref}

\usepackage[all,ps,cmtip]{xy}

\DeclareRobustCommand{\SkipTocEntry}[5]{}

\def\R{{\bf R}}

\def\llra{\hbox to 10mm{\rightarrowfill}}

\def\lllra{\hbox to 15mm{\rightarrowfill}}

\def\PZ{{\widehat Z}}

\def\phi{{\varphi}}

\def\wK{{\widetilde K}}

\def\wY{{\widetilde Y}}

\def\cO{\mathcal{O}}

\DeclareMathOperator{\Pic}{Pic}

\newtheorem{lemm}{Lemma}[section]
\newtheorem{theo}[lemm]{Theorem}

\newtheorem{prop}[lemm]{Proposition}

\theoremstyle{definition}

\newtheorem{rema}[lemm]{Remark}

\newtheorem{qu}[lemm]{Question}

\theoremstyle{remark}
\newtheorem*{remark*}{Remark}
\newtheorem*{note*}{Note}

\begin{document}
\title{On the Chow ring of certain rational cohomology tori}

\author{Zhi Jiang}
\address{D\'epartement de Math\'ematiques\\CNRS UMR 8628\\Universit\'e Paris-Sud\\B\^atiment 425\\91405 Orsay Cedex\\France}
\email{zhi.jiang@math.u-psud.fr}

\author{Qizheng Yin}
\address{Departement Mathematik\\ETH Z\"urich\\R\"amistrasse 101\\8092 Z\"urich\\Switzerland}
\email{qizheng.yin@math.ethz.ch}



\thanks{Q.~Y.~was supported by the grant ERC-2012-AdG-320368-MCSK}

\begin{abstract}
Let $f: X\rightarrow A$ be an abelian cover from a complex algebraic variety with quotient singularities to an abelian variety. We show that $f^*$ induces an isomorphism between the rational cohomology rings $H^{\bullet}(A, \mathbb{Q})$ and $H^{\bullet}(X, \mathbb{Q})$ if and only if $f^*$ induces an isomorphism between the Chow rings with rational coefficients $\mathrm{CH}^\bullet(A)_{\mathbb{Q}}$ and  $\mathrm{CH}^\bullet(X)_{\mathbb{Q}}$.
\end{abstract}

\maketitle

\section{Rational cohomology tori and the Chow ring}
A complex algebraic variety $X$ is called a rational cohomology torus if $X$ is normal and
$$H^{\bullet}(X, \mathbb{Q})\simeq \wedge^{\bullet}H^1(X, \mathbb{Q}).$$
In \cite{DJL}, the authors studied properties of rational cohomology tori. They showed that if $X$ is a rational cohomology torus, then there exists a finite cover $a_X: X\rightarrow A$ to an abelian variety such that $a_X^*: H^{\bullet}(A, \mathbb{Q})\rightarrow H^{\bullet}(X, \mathbb{Q})$ is an isomorphism. The morphism $a_X$ is the universal morphism to an abelian variety and is called the Albanese morphism of $X$.

It is a general principle that Hodge structures control Chow groups. Hence it makes sense to ask the following question.

\begin{qu} Let $f: X \rightarrow A$ be a finite morphism from a variety with quotient singularities to an abelian variety.
Assume that $f^*: H^{\bullet}(A, \mathbb{Q})\rightarrow H^{\bullet}(X, \mathbb{Q})$ is an isomorphism. 
Does $f^*$ induce an isomorphism of Chow rings with rational coefficients between $A$ and $X$?
\end{qu}

A variety is said to have quotient singularities if it is \'etale-locally the quotient of a smooth variety by a finite group. For such varieties there is a well-defined intersection theory with $\mathbb{Q}$-coefficients (see \cite{V}). Note that a smooth rational cohomology torus cannot be of general type, but in every dimension $>2$, there exist rational cohomology tori of general type with quotient singularities (see \cite{DJL}). Hence it is more interesting to consider the Chow ring of rational cohomology tori with quotient singularities.

Despite several constraints on $a_X$ found in \cite{DJL}, we do not have a thorough classification of rational cohomology tori. Hence it seems difficult to answer the above question in general. But all known rational cohomology tori are actually abelian covers of abelian varieties (see \cite[Section 4]{DJL}). The goal of this note is to answer affirmatively the above question in the case of abelian covers of abelian varieties.

Let $f: X\rightarrow A$ be an abelian cover from a normal variety to an abelian variety with Galois group $G$. Then we have (see for instance \cite{P})
$$f_*\cO_X=\bigoplus_{\chi \in G^*} L_{\chi}^{-1},$$
where $L_{\chi}$ is a line bundle on $A$ and $G$ acts on $L_{\chi}^{-1}$ via the character $\chi$. Assume further that $X$ has quotient singularities. Define for each character $\chi \in G^*$ the projector
$$[\Gamma_\chi^X] := \frac{1}{|G|}\sum_{g\in G}\chi^{-1}(g)[\Gamma_g^X] \in \mathrm{CH}^{\dim X}(X\times X)_{\mathbb{C}},$$
where $\Gamma^X_g$ is the graph of $g$. The image of $\mathrm{CH}^\bullet(X)_{\mathbb{C}}$ under $[\Gamma_\chi^X]$ is
$$\mathrm{CH}^\bullet(X)_{\mathbb{C}}^{\chi} := \{ \alpha \in \mathrm{CH}^\bullet(X)_{\mathbb{C}} \,:\, g_*(\alpha) = \chi(g) \alpha \, \text{ for all } \, g \in G\}.$$
We have a decomposition
$$\mathrm{CH}^\bullet(X)_{\mathbb{C}} = \bigoplus_{\chi \in G^*} \mathrm{CH}^\bullet(X)_{\mathbb{C}}^{\chi}.$$

\begin{theo}\label{main}
Let $f: X\rightarrow A$ be an abelian cover from a variety with quotient singularities to an abelian variety with Galois group $G$. Then the following statements are equivalent:
\begin{enumerate}
\item $X$ is a rational cohomology torus;
\item $h^i(A, L_{\chi})=0$ for all $i\in \mathbb{Z}$ and non-trivial $\chi \in G^*$;
\item $[\Gamma_\chi^X] = 0 \in \mathrm{CH}^{\dim X}(X\times X)_{\mathbb{C}}$ for all non-trivial $\chi \in G^*$; in particular, 
$$f^*: \mathrm{CH}^\bullet(A)_{\mathbb{Q}}\rightarrow \mathrm{CH}^\bullet(X)_{\mathbb{Q}}$$
is an isomorphism.
\end{enumerate}
\end{theo}

The implications (1) $\Rightarrow$ (2) and (3) $\Rightarrow$ (1) are straightforward. Indeed, if $X$ is a rational cohomology torus, then $f^*: H^\bullet(A, \mathbb{Q})\rightarrow H^\bullet(X, \mathbb{Q})$ is an isomorphism of $\mathbb{Q}$-Hodge structures. We have for all $i \in \mathbb{Z}$,
$$h^i(A, \cO_A) = h^{0, i}(A) = h^{0, i}(X) = h^i(X, \cO_X) = \sum_{\chi\in G^*}h^i(A, L_{\chi}^{-1}).$$
Hence $h^i(A, L_{\chi}^{-1})=0$ for all $i\in \mathbb{Z}$ and non-trivial $\chi \in G^*$. We also have $h^i(A, L_{\chi})=0$ by Serre duality.

If $[\Gamma_\chi^X] = 0 \in \mathrm{CH}^{\dim X}(X\times X)_{\mathbb{C}}$ for all non-trivial $\chi \in G^*$, then by applying $[\Gamma_\chi^X]$ to $H^\bullet(X, \mathbb{C})$, we find $H^\bullet(X, \mathbb{C}) = H^\bullet(X, \mathbb{C})^G$. Hence $f^* : H^\bullet(A, \mathbb{Q}) \to H^\bullet(X, \mathbb{Q})$ is an isomorphism and $X$ is a rational cohomology torus.

The proof of (2) $\Rightarrow$ (3) will occupy the next section.

\begin{rema}
Since $L_{\chi}$ is semi-ample (see Step 1 in the proof of Proposition \ref{triviality}), condition~(2) in Theorem \ref{main} is also equivalent to the condition that $h^0(A,L_{\chi})=0$ for all non-trivial $\chi\in G^*$. Hence, under the assumption of Theorem \ref{main}, we know that $X$ is a rational cohomology torus if and only if $p_g(X)=1$.
\end{rema}
 
\section{Proof of Theorem \ref{main}}

\begin{lemm}\label{isogeny}
Let $A$ be an abelian variety and let $t: A\rightarrow A$ be the translation by a torsion element. Then $[\Gamma_t^A]-[\Delta_A]=0\in \mathrm{CH}^{\dim A}(A\times A)_{\mathbb{Q}}$, where $\Gamma_t^A$ is the graph of $t$ and $\Delta_A$ is the diagonal.
\end{lemm}

\begin{proof}
Let $m\in \mathbb{Z}_{>0}$ be the order of $t$. Recall that $m^*: \mathrm{CH}^\bullet(A\times A)_{\mathbb{Q}}\rightarrow \mathrm{CH}^\bullet(A\times A)_{\mathbb{Q}}$ is an isomorphism (see \cite{B}). Since $m^*[\Gamma_t^A]=m^*[\Delta_A]\in  \mathrm{CH}^{\dim A}(A\times A)_{\mathbb{Q}}$, we conclude that $[\Gamma_t^A]-[\Delta_A]=0\in \mathrm{CH}^{\dim A}(A\times A)_{\mathbb{Q}}$.
\end{proof}
 
\begin{prop}\label{triviality}
Let $\rho: Y\rightarrow A$ be a cyclic cover from a variety with quotient singularities to an abelian variety with Galois group $H$. We write 
$$\rho_*\cO_Y=\bigoplus_{\chi\in H^*}L_{\chi}^{-1}.$$
If $h^i(A, L_{\chi})=0$ for all $i\in \mathbb{Z}$ and non-trivial $\chi \in H^*$, then for all non-trivial $\chi \in H^*$,
$$\sum_{h\in H}\chi^{-1}(h)[\Gamma_h^Y]=0\in  \mathrm{CH}^{\dim Y}(Y\times Y)_{\mathbb{C}}.$$
\end{prop}
\begin{proof}
We proceed by induction on the dimension of $Y$. The statement in dimension $0$ is trivial. We then assume that Proposition \ref{triviality} holds in dimensions $<\dim Y$.

\vspace{.5\baselineskip plus .2\baselineskip minus .2\baselineskip}
\noindent {\bf Step 1.} {\em Reduction to lower dimensions.}

\vspace{.5\baselineskip plus .2\baselineskip minus .2\baselineskip}
Let $\chi$ be a generator of $H^*$. Since $h^i(A, L_{\chi})=0$ for all $i\in \mathbb{Z}$, the line bundle $L_{\chi}$ is not ample. Let $K$ denote the neutral component of the morphism $\varphi_{L_\chi}: A\rightarrow \Pic^0(A)$ induced by $L_\chi$, and let
$p : A\rightarrow B$ be the quotient by $K$. Then we have
$$L_\chi=p^*H + Q,$$
where $H$ is an ample line bundle on $B$ and $Q$ is a torsion line bundle such that $Q|_K\neq \cO_K$ (see for instance \cite[Section~3.3]{BL}). As $\rho$ is a cyclic cover and $\chi$ is a generator of~$H^*$, for each $k\in \mathbb{Z}$, there exists an effective divisor $D_k$ whose components are irreducible components of the branched divisor $D$ such that
$$kL_{\chi}=L_{\chi^k}+D_{k}.$$
Moreover, for $m=|H|$, we know that $mL_\chi=D_m$ and that $D_m-D$ is effective (see \cite{P}). In particular, we have $D=p^*D_B$ for some ample divisor $D_B$ on $B$.
 
Consider the Stein factorization of the natural morphism $\varphi = p \circ \rho: Y\rightarrow B$,
$$\varphi: Y\xrightarrow{\varphi'} Z\xrightarrow{\varphi''} B.$$  
A general fiber of $\varphi'$ is an \'etale cover of $K$, which is an abelian variety that we denote by~$\wK$. Hence $\varphi'$ is the Iitaka fibration of $Y$ in the sense of \cite[Theorem 13]{kaw}.
 
Let $Y':=Z\times_BA$. Then we have a factorization
$$\rho: Y\xrightarrow{\rho'} Y'\xrightarrow{\rho''} A.$$
Both $\rho'$ and $\rho''$ are cyclic covers with respective Galois groups $H'\subset H$ and $H'':=H/H'$. Hence the morphism $\varphi'': Z\rightarrow B$ is also an $H''$-cover. Moreover, we write 
$$\varphi''_*\cO_Z=\bigoplus_{\iota\in H''^*}L_{\iota, Z}^{-1}.$$
Then we have $p^*L_{\iota, Z}=L_{\chi_\iota}$, where $\chi_\iota$ is the induced character of $H$. Hence $H^i(B, L_{\iota, Z})=0$ for all $i\in \mathbb{Z}$ and non-trivial $\iota\in H''^*$.

Further, we show that $Y'$ and $Z$ have quotient singularities. Recall from \cite[Proposition~2.8]{V} that $Y$ is the coarse moduli space of a smooth Deligne-Mumford stack $\mathcal{Y}$ (called the canonical stack). The morphism $\mathcal{Y} \to Y$ is an isomorphism away from a locus of codimension $\geq 2$. Since $Y$ is normal, the action of~$H'$ on $Y$ induces an action on $\mathcal{Y}$. The quotient $Y' = Y/H'$ is the coarse moduli space of the smooth Deligne-Mumford stack $[\mathcal{Y}/H']$, and hence has quotient singularities. 

We take a quotient $A\rightarrow K$ such that $A\rightarrow B\times K$ is an isogeny. Each connected component of a general fiber of  the induced morphism $Y'\rightarrow K$ is an \'etale cover of $Z$. Hence $Z$ also has quotient singularities.

Since $\dim Z < \dim Y$, the induction hypothesis implies that for all non-trivial $\iota\in H''^*$,
$$\sum_{h\in H''}\iota^{-1}(h)[\Gamma_h^Z]=0\in \mathrm{CH}^{\dim Z}(Z\times Z)_{\mathbb{C}}.$$

\noindent {\bf Step 2.} {\em Proof that $\sum_{h\in H''}\iota^{-1}(h)[\Gamma_h^{Y'}]=0\in \mathrm{CH}^{\dim Y'}(Y'\times Y')_{\mathbb{C}}$ for all non-trivial $\iota\in H''^*$.}

\vspace{.5\baselineskip plus .2\baselineskip minus .2\baselineskip}
Consider the quotient $A\rightarrow K$ such that $A\rightarrow B\times K$ is an isogeny with Galois group $L$. Then the induced morphism $\pi: Y'=Z\times_BA\rightarrow Z\times K$ is also an \'etale cover with Galois group $L$.

Note that for all $h\in H''$,
$$(\pi\times \pi)^*[\Gamma^Z_h\times \Delta_K]=\sum_{l\in L}[\Gamma^Z_h\times_{\Delta_B}\Gamma^A_l ]\in \mathrm{CH}^{\dim Y'}(Y'\times Y')_{\mathbb{Q}}.$$
Consider for each $l\in L$ the embeddings $\Gamma^Z_h\times_{\Delta_B}\Gamma^A_l \subset \Gamma^{Z}_h\times_{\Delta_B}(A\times_BA) \subset Y'\times Y'$. By Lemma~\ref{isogeny}, we have $[\Gamma^A_l]=[\Delta_A]\in \mathrm{CH}_{\dim Y'}(A\times_BA)_{\mathbb{Q}}$. Hence $$[\Gamma^Z_h\times_{\Delta_B}\Gamma^A_l]=[\Gamma^Z_h\times_{\Delta_B}\Delta_A]\in \mathrm{CH}_{\dim Y'}(\Gamma^{Z}_h\times_{\Delta_B}(A\times_BA))_{\mathbb{Q}}.$$
In particular, $$[\Gamma^Z_h\times_{\Delta_B}\Gamma^A_l]=[\Gamma^Z_h\times_{\Delta_B}\Delta_A]\in \mathrm{CH}^{\dim Y'}(Y'\times Y')_{\mathbb{Q}}.$$
 
Then we find
\begin{align*}
(\pi\times \pi)^*\big(\sum_{h\in H''}\iota^{-1}(h)[\Gamma^Z_h\times\Delta_K]\big) &= \sum_{h\in H''}\iota^{-1}(h)\big(\sum_{l\in L} [\Gamma^Z_h\times_{\Delta_B}\Gamma^A_l]\big)\\
&= |L|\sum_{h\in H''}\iota^{-1}(h)[\Gamma^Z_h\times_{\Delta_B}\Delta_A]\\
&= |L|\sum_{h\in H''}\iota^{-1}(h)[\Gamma^{Y'}_h] \in \mathrm{CH}^{\dim Y'}(Y'\times Y')_{\mathbb{C}}.
\end{align*}
As $\sum_{h\in H''}\iota^{-1}(h)[\Gamma^Z_h\times\Delta_K]=0\in \mathrm{CH}^{\dim Y'}(Z \times Z \times K \times K)_{\mathbb{C}}$, we also have
$$\sum_{h\in H''}\iota^{-1}(h)[\Gamma^{Y'}_h] =0\in \mathrm{CH}^{\dim Y'}(Y'\times Y')_{\mathbb{C}}.$$

\noindent {\bf Step 3.} {\em  Proof that $\sum_{h\in H'}\chi^{-1}(h)[\Gamma_h^Y]=0\in \mathrm{CH}^{\dim Y}(Y\times Y)_{\mathbb{C}}$ for all non-trivial $\chi\in H'^*$.}

\vspace{.5\baselineskip plus .2\baselineskip minus .2\baselineskip}
By Kawamata's theorem \cite[Theorem 13]{kaw}, there exists an \'etale cover $\mu: \wY\rightarrow Y$ with Galois group $J$ induced by an \'etale cover of $A$, such that the following holds: consider the Stein factorization of the natural morphism $\nu = \varphi' \circ \mu: \wY\rightarrow Z$,
$$\nu: \wY\xrightarrow{\nu'} \PZ\xrightarrow{\nu''} Z.$$
Then we have $\wY\simeq \PZ\times \wK$ with $\nu'$ the projection, $J$ acts on $\PZ$ faithfully and on $\wK$ faithfully by translation, and $Y\simeq (\PZ\times \wK)/J$. We put everything into a commutative diagram:
$$\xymatrix{
\wY\ar[d]^{\nu'} \ar[r]^{\mu} &Y \ar@/^2pc/[rr]^{\rho}\ar[r]^{\rho'}\ar[d]^{\varphi'} & Y'\ar[r]^{\rho''}\ar[dl] & A\ar[dl]^p\\
\PZ\ar[r]^{\nu''}& Z\ar[r]^{\varphi''} & B.}$$

Since $\mu$ is induced by an \'etale cover of $A$, we know that $\rho'\circ\mu: \wY\rightarrow Y'$ is still an abelian cover. It is clear that the Galois group of $\rho\circ\mu$ is $J\times H$ and hence the Galois group of $\rho'\circ \mu$ is $J\times H'$. Let $q: J\times H' \rightarrow H'$ be the projection. 

Consider the morphism $\mu\times\mu: \wY\times \wY\rightarrow Y\times Y$. Note that for all $h\in H'$,
$$(\mu\times\mu)^*[\Gamma_h^Y]=\sum_{g\in q^{-1}(h)}[\Gamma_g^{\wY}] \in \mathrm{CH}^{\dim \wY}(\wY \times \wY)_{\mathbb{Q}}.$$
Hence for all non-trivial $\chi\in H'^*$,
$$(\mu\times\mu)^*(\sum_{h\in H'}\chi^{-1}(h)[\Gamma_h^Y]) = \sum_{(g, h)\in J\times H'} \chi^{-1}(h)[\Gamma_{(g, h)}^{\wY}] \in \mathrm{CH}^{\dim Y}(\wY\times\wY)_{\mathbb{C}}.$$

Under the isomorphism $\wY \simeq \PZ\times \wK$, we have 
$$[\Gamma_{(g, h)}^{\wY}]=[\Gamma^{\PZ}_g\times \Gamma_{(g, h)}^{\wK}]\in \mathrm{CH}^{\dim Y}(\PZ\times \PZ\times \wK \times \wK)_{\mathbb{Q}}.$$
Then we find 
\begin{align*}
(\mu\times\mu)^*(\sum_{h\in H'}\chi^{-1}(h)[\Gamma_h^Y])
&= \sum_{(g,h)\in J\times H'} \chi^{-1}(h)[\Gamma^{\PZ}_g\times \Gamma_{(g, h)}^{\wK}]\\
&= \sum_{g\in J} ([\Gamma^{\PZ}_g]\times (\sum_{h\in H'}\chi^{-1}(h)[\Gamma_{(g, h)}^{\wK}]))\in \mathrm{CH}^{\dim Y}(\PZ\times \PZ\times \wK\times \wK)_{\mathbb{C}}.
\end{align*}

By Lemma \ref{isogeny}, we have $[\Gamma_{(g, h)}^{\wK}]=[\Delta_{\wK}]\in \mathrm{CH}^{\dim \wK}(\wK \times \wK)_{\mathbb{Q}}$. Hence
$$\sum_{h\in H'}\chi^{-1}(h)[\Gamma_{(g, h)}^{\wK}]=0\in \mathrm{CH}^{\dim \wK}(\wK \times \wK)_{\mathbb{C}},$$
and thus $(\mu\times\mu)^*(\sum_{h\in H'}\chi^{-1}(h)[\Gamma_h^Y])=0\in \mathrm{CH}^{\dim Y}(\wY\times\wY)_{\mathbb{C}}$. By the projection formula,
$$\sum_{h\in H'}\chi^{-1}(h)[\Gamma_h^Y]=0\in \mathrm{CH}^{\dim Y}(Y\times Y)_{\mathbb{C}}.$$

\noindent {\bf Step 4.} {\em Conclusion.}

\vspace{.5\baselineskip plus .2\baselineskip minus .2\baselineskip}
Let $\chi \in H^*$ be a non-trivial character. If the restriction of $\chi$ to $H'$ is trivial, then $\chi$ is the pull-back of a non-trivial character $\iota$ of $H''$. As before,
$$\sum_{h\in H}\chi^{-1}(h)[\Gamma^Y_h]=(\rho'\times\rho')^*(\sum_{h\in H''}\iota^{-1}(h)[\Gamma^{Y'}_h])\in \mathrm{CH}^{\dim Y}(Y\times Y)_{\mathbb{C}}.$$
By Step 2, we see that $\sum_{h \in H}\chi^{-1}(h)[\Gamma^Y_h] = 0 \in \mathrm{CH}^{\dim Y}(Y\times Y)_{\mathbb{C}}$.

If the restriction of $\chi$ to $H'$ is non-trivial, we take a subset $V$ of $H$ consisting of representatives of each coset of $H'$.
Then we have 
$$\sum_{h\in H}\chi^{-1}(h)[\Gamma^Y_h] = (\sum_{h\in H'}\chi^{-1}(h)[\Gamma^Y_{h}])\circ(\sum_{g\in V}\chi^{-1}(g)[\Gamma^Y_{g}])\in \mathrm{CH}^{\dim Y}(Y\times Y)_{\mathbb{C}}.$$
By Step 3, we also know that $\sum_{h \in H}\chi^{-1}(h)[\Gamma^Y_h] = 0 \in \mathrm{CH}^{\dim Y}(Y\times Y)_{\mathbb{C}}$.
\end{proof}

We are now ready to prove Theorem \ref{main}. Let $\chi \in G^*$ be a non-trivial character. Consider
$$G' := \mathrm{Ker}(\chi: G \to \mathbb{C}^*) \subset G.$$
We take $X' := X/G'$ which has quotient singularities (see Step 1 in the proof of Proposition~\ref{triviality}), and we have a factorization
$$f : X \xrightarrow{f'} X' \xrightarrow{f''} A.$$
Then $f''$ is a cyclic cover with Galois group $G'' := G / G' \subset \mathbb{C}^*$.  Moreover, the character $\chi$ is the pull-back of a non-trivial character $\iota$ of $G''$. It is straightforward to verify that the assumption of Proposition~\ref{triviality} is satisfied by $X'$ and $G''$. Hence
$$\sum_{g \in G''} \iota^{-1}(g)[\Gamma_g^{X'}] = 0 \in \mathrm{CH}^{\dim X'}(X' \times X')_{\mathbb{C}}.$$
As before,
$$\sum_{g \in G} \chi^{-1}(g)[\Gamma_g^X] = (f' \times f')^*(\sum_{g \in G''} \iota^{-1}(g)[\Gamma_g^{X'}]) = 0 \in \mathrm{CH}^{\dim X}(X \times X)_{\mathbb{C}},$$
and thus
\[\pushQED{\qed}[\Gamma_{\chi}^X] = \frac{1}{|G|}\sum_{g \in G} \chi^{-1}(g)[\Gamma_g^X] = 0 \in \mathrm{CH}^{\dim X}(X \times X)_{\mathbb{C}}.\qedhere\popQED\]

\section*{Acknowledgements}

We thank Barbara Fantechi, Lie Fu, Johannes Schmitt, and Zhiyu Tian for helpful discussions.

\end{document}